\documentclass[a4paper, 12pt]{amsart}
\usepackage[T1]{fontenc}
\usepackage{amssymb}
\usepackage[utf8]{inputenc}
\usepackage[pagewise]{lineno}
\usepackage{multicol}
\usepackage{hyperref}
\usepackage{bm}
\usepackage{comment}
\usepackage{xcolor}
\usepackage{mathrsfs}
\usepackage{cleveref}
\usepackage{amsmath}
\usepackage{longtable}

\newcommand{\supp}{\mbox{Supp}}
\newcommand{\cd}{\mathcal{D}}
\newcommand{\mcr}{\mathcal{R}}
\newcommand{\F}{\mathbb{F}}

\makeatletter
\@namedef{subjclassname@2020}{%
  \textup{2020} Mathematics Subject Classification}
\makeatother

\newtheorem{theorem}{Theorem}[section]
\newtheorem{proposition}[theorem]{Proposition}
\newtheorem{lemma}[theorem]{Lemma}
\newtheorem{corollary}[theorem]{Corollary}
\theoremstyle{definition}
\newtheorem{remark}[theorem]{Remark}
\newtheorem{definition}[theorem]{Definition}
\newtheorem{example}[theorem]{Example}

\numberwithin{equation}{section}

\begin{document}

\title{On homogeneous involutions on matrix algebras}

\author[Garcia]{Micael Said Garcia}
\address{Department of Mathematics,
Instituto de Matemática, Estatística e Ciência da Computação (IME) - Universidade de São Paulo (USP), São Paulo - Brazil}
\email{micael.said@usp.br}
\author[Sampaio]{Cassia Ferreira Sampaio}
\address{Faculdade de Tecnologia do Estado de São Paulo (FATEC), São Paulo - Brazil}
\email{cassia.sampaio@fatec.sp.gov.br}

\thanks{M. Garcia was financed in part by the Coordenação de Aperfeiçoamento de Pessoal de Nível Superior - Brasil (CAPES) - Finance Code 001.}



    \begin{abstract}
        We study the homogeneous involutions on the full square matrices over an algebraically closed field endowed with a division grading with commutative support. We obtain the classification of the isomorphism and equivalence classes for the Pauli grading. We also investigate the homogeneous involutions on the full square matrices with entries in a finite-dimensional graded-division algebra over an algebraically closed field of characteristic not $2$ endowed with an arbitrary grading by an arbitrary group.
    \end{abstract}
    \maketitle

    \section*{Introduction}
    Graded algebras appear naturally in several branches of Mathematics, often serving as a fundamental tool for understanding an object. Moreover, the study of graded algebras is also important by itself. In a given class of algebras, two important problems in this sense are the classification of the isomorphism classes of group gradings, and the classification of the equivalence classes of fine group gradings.

    In this direction, the gradings  on matrix algebras were described in \cite{BZ02}. In \cite{BSZ05}, the authors described the group gradings by a finite Abelian group on several types of simple Jordan and Lie algebras over an algebraically closed field of characteristic $0$. In their approach, the so-called \textit{graded involutions} (or \textit{degree-preserving involutions}) on matrix algebras played an essential role. Later, the classification of the isomorphism classes of graded involutions on matrix algebras graded by an abelian group and of the group gradings on classical simple Lie Algebras over an algebraically closed field of characteristic different from $2$ was given in \cite{BK10,EK15}. Also, the classification of the equivalence classes of fine group gradings on classical simple Lie Algebras over an algebraically closed field was first found in \cite{E10} assuming characteristic $0$ and later generalized in \cite{EK2013,EK15} for characteristic different from $2$.
    

    The same kind of problem can also be asked in the context of non-simple algebras. For example, the group gradings on upper triangular matrices were classified in \cite{VKV04, VZ07}. In \cite{VZ09}, their graded involutions over an algebraically closed field of characteristic zero were classified when the grading group is finite and abelian. Later, in \cite{DRG24}, the classification of the graded involutions on upper triangular matrices was given over an arbitrary field of characteristic not $2$ and an arbitrary group. In the last paper, the authors also studied their graded identities with involution for the finest grading.

    In other direction, \textit{degree-inverting involutions}, i.e., involutions that act on homogeneous components by inverting degrees appear naturally in various contexts. For example, the transposition on a matrix algebra with an elementary grading and the usual involution on Leavitt path algebras with its natural $\mathbb{Z}$-grading invert the degrees (see \cite{H2016}). Over a field of characteristic not $2$, the work \cite{FSY22} describes the degree-inverting involutions on upper triangular matrices and also on full square matrix algebra assuming further that the ground field is algebraically closed.


    A natural generalization of both graded and degree-inverting involutions is given by involutions that send homogeneous components to homogeneous components. This type of involution is called a \textit{homogeneous involution}, and it was first considered in \cite{M22}, where the author described it on upper triangular matrices over a field of characteristic different from $2$. Homogeneous involutions on finite-dimensional graded-division algebras over an algebraically closed field and their polynomial identities were also described in \cite{Y24}. Similarly, the star-homogeneous-graded identities on upper triangular matrices were also studied in \cite{MY25,DY25}.
    
    The purpose of this paper is to study the homogeneous involutions on matrix algebras over an algebraically closed field of characteristic not $2$. The paper is divided as follows: Section 1 is dedicated to some preliminaries. Then, we provide some partial results for homogeneous involutions on matrix algebras graded with a division grading by an abelian group in Section 2. Finally, in Section 3, we develop a similar theory as presented in \cite[Section 2.4]{EK2013} and \cite[Section 4]{FSY22} to study group gradings on matrix algebras with entries in a finite-dimensional graded-division algebra.    
    

    \section{Preliminaries}    
    \subsection{Graded algebras}
    Let $G$ be a group and $\mathcal{A}$ be an algebra. We use the multiplicative notation for $G$ and denote its neutral element by $1$. A $G$-grading $\Gamma$ on $\mathcal{A}$ is a vector space decomposition
    $$\Gamma\colon \mathcal{A}=\bigoplus_{g\in G}\mathcal{A}_{g}$$
    such that $\mathcal{A}_{g}\mathcal{A}_{h}\subseteq \mathcal{A}_{gh}$ for all $g, h \in G$. Each subspace $\mathcal{A}_{g}$ is called \emph{homogeneous component of degree $g$}. A nonzero element $x \in \mathcal{A}_{g}$ is called a \emph{homogeneous element of degree $g$}, and we denote $\deg x=g$. An algebra $\mathcal{A}$ endowed with a $G$-grading $\Gamma$ is called a \emph{$G$-graded algebra}. The \emph{support} of the grading $\Gamma$ is the set $\mathrm{Supp}\,\Gamma=\{g\in G\mid\, \mathcal{A}_{g}\neq 0\}$. When the grading is fixed, we shall denote the support by $\mathrm{Supp}\,\mathcal{A}$. A \emph{graded-division algebra} $\mathcal{D}$ is a unital associative algebra, where each nonzero homogeneous element is invertible.

    If $\mathcal{A}$ is a $G$-graded algebra and $\mathcal{B}$ is an $H$-graded algebra, an isomorphism of algebras $f\colon \mathcal{A}\to \mathcal{B}$ is called an \emph{equivalence} if there exists a bijection $\alpha\colon \supp\,\mathcal{A}\to \supp\,\mathcal{B}$ such that $f(\mathcal{A}_{g})=\mathcal{B}_{\alpha(g)}$ for all $g \in \supp\,\mathcal{A}$. When $G=H$ and $\alpha$ is the identity map, $f$ is called a \emph{graded isomorphism}. In this case, we say that $\mathcal{A}$ and $\mathcal{B}$ are isomorphic as $G$-graded algebras.

    Let $\mathcal{A}=\bigoplus_{g\in G}\mathcal{A}_{g}$ be a $G$-graded algebra. Then, a \emph{$G$-graded right $\mathcal{A}$-module} is a right $\mathcal{A}$-module $V$ endowed with a vector space decomposition $V=\bigoplus _{g\in G}V_{g}$ such that $V_{h}\mathcal{A}_{g}\subseteq V_{hg}$ for all $h,g \in G$. We say that $V$ is a \emph{graded-simple} right $\mathcal{A}$-module if $V\mathcal{A}\neq0$ and the only graded submodules of $V$ are $0$ and $V$. Given $V$ and $W$ two $G$-graded right $\mathcal{A}$-modules and $g \in G$, we say that $f\colon V\to W$ is a \emph{graded map of degree $g$} if $f$ is a homomorphism of $\mathcal{A}$-modules and $f(V_{h})\subseteq W_{gh}$ for all $h \in G$. We denote by $\mathrm{Hom}_{g}(V, W)$ the vector space of all graded maps of degree $g$. Then, we define $\mathrm{Hom}^{gr}(V,W)=\bigoplus _{g\in G}\mathrm{Hom}_{g}(V,W)$. If $V$ is finite-dimensional, then we have $\mathrm{Hom}_{\mathcal{A}}(V,W)=\mathrm{Hom}^{gr}(V,W)$. We can also define $G$-graded left $\mathcal{A}$-modules and repeat this discussion for them, making some adjustments (see \cite[Section 2.1]{EK2013} for further details).
    
    Now, let $\mathcal{D}$ be a finite-dimensional $G$-graded-division algebra, and let $V$ be a finite-dimensional $G$-graded right $\mathcal{D}$-module. Then, $\mathcal{R}=\mathrm{End}_{\mathcal{D}}(V)=\mathrm{Hom}_{\mathcal{D}}(V,V)$ is a $G$-graded algebra isomorphic to a matrix algebra with entries in $\cd$. Furthermore, $V$ is a $G$-graded left $\mathcal{R}$-module.

    Taking $\mcr=M_{n}(\F)$ a matrix algebra endowed with a $G$-grading, we can find a finite-dimensional graded-division algebra $\cd$ and a sequence $\gamma=(g_{1},\ldots,g_{m})\in G^{m}$ such that $\mcr\cong M_{m}(\F)\otimes \cd$, where the grading is given by
    $$\deg e_{ij}\otimes d=g_{i}(\deg d)g_{j}^{-1}\quad d\in \cd \text{ homogeneous}.$$
    Actually, given a graded-division algebra $\cd$, we can always endow $M_{m}(\F)\otimes \cd$ with a grading of this kind. We shall denote this kind of grading in $M_{m}(\F)\otimes \cd$ by $\mathcal{M}(\mathcal{D},\gamma)$.

    \subsection{Homogeneous anti-automorphisms}
    Let $\mathcal{A}$ be a $G$-graded algebra, and let $\tau\colon \mathrm{Supp}\,\mathcal{A}\to \mathrm{Supp}\,\mathcal{A}$ be a bijection. An algebra automorphism $\varphi\colon \mathcal{A}\to \mathcal{A}$ is called a \emph{$\tau$-homogeneous automorphism} if  $\varphi(\mathcal{A}_{g})=\mathcal{A}_{\tau(g)}$ for all $g \in \mathrm{Supp}\,\mathcal{A}$.  Similarly, an $\F$-linear anti-\linebreak automorphism $\psi\colon \mathcal{A}\to \mathcal{A}$ is called a \emph{$\tau$-homogeneous anti-automorphism} if $\psi(\mathcal{A}_{g})=\mathcal{A}_{\tau(g)}$ for all $g \in \mathrm{Supp}\,\mathcal{A}$. If $\psi$ is also an involution, then we say that $\psi$ is a \emph{$\tau$-homogeneous involution}. Notice that we can always extend (not uniquely) $\tau$ to be a bijection on $G$ (for example, sending $g$ to $g$ for all $g \notin \mathrm{Supp}\,\mathcal{A}$). We say that an automorphism or an anti-automorphism is \emph{homogeneous} if it is $\tau$-homogeneous for some bijection $\tau\colon G\to G$.

    \begin{lemma}
        \label{lemmaTau}
        Let $T$ be a group, $\tau\colon T\to T$ be a bijection, and $\cd=\bigoplus _{t\in T}\cd_{t}$ be a graded-division algebra with support $T$. If there exists a $\tau$-homogeneous anti-automorphism $\psi$ on $\cd$, then for all $g,h \in T$ we have
        $$\tau(gh)=\tau(h)\tau(g).$$
        If $\psi$ is also an involution, then
        \[
            \pushQED{\qed} 
            \tau(\tau(g))=g,\quad \forall g\in T.\qedhere
            \popQED
        \]     
    \end{lemma}
    
    Let $\mathcal{A}$ and $\mathcal{B}$ be two $G$-graded algebras. Let $\psi$ be a homogeneous anti-automorphism on $\mathcal{A}$ and $\psi'$ be a homogeneous anti-automorphism on $\mathcal{B}$. We say that the pairs $(\mathcal{A},\psi)$ and $(\mathcal{B},\psi')$ are \emph{equivalent} if there exists an equivalence $\varphi\colon \mathcal{A}\to\mathcal{B}$ such that $\varphi\circ \psi=\psi'\circ\varphi$. Analogously, we say that $(\mathcal{A},\psi)$ and $(\mathcal{B},\psi')$ are \emph{isomorphic} if there exists a graded isomorphism $\varphi\colon \mathcal{A}\to\mathcal{B}$ such that $\varphi\circ \psi=\psi'\circ\varphi$. In the case where $\mathcal{A}=\mathcal{B}$, we say that $\psi$ and $\psi'$ are equivalent (or isomorphic).  

    \subsection{Factor sets}
    Let $T$ be a finite group and $\F^{\times}$ be the set of invertible elements of $\F$. A map $\sigma\colon T\times T\to \F^{\times}$ is called a \emph{$2$-cocycle} or a \emph{factor set} if
    $$\sigma(u,v)\sigma(uv,w)=\sigma(u,vw)\sigma(v,w),\quad\forall u,v,w \in T.$$
    We denote by $Z^{2}(T,\F^{\times})$ the set of all factor sets $\sigma\colon T\times T\to \F^{\times}$. It is well-known that $Z^{2}(T,\F^{\times})$ with the natural product forms an abelian group.

    Given $\sigma\in Z^{2}(T,\F^{\times})$, we define the \emph{twisted group algebra $\F^{\sigma}T$ with respect to $\sigma$} being the algebra with basis $\{X_{t}\mid t\in T\}$ and product given by
    $$X_{u}X_{v}=\sigma(u,v) X_{uv},\quad u,v\in T.$$
    The fact that $\sigma \in Z^{2}(T,\F^{\times})$ implies that $\F^{\sigma}T$ is an associative algebra. Moreover, $\F^{\sigma}T$ has a natural grading that makes $\F^{\sigma}T$ a graded-division algebra.

    For any map $\lambda\colon T\to \F^{\times}$, we obtain a factor set $\delta\lambda$ given by
    $$\delta\lambda(u,v)=\frac{\lambda(u)\lambda(v)}{\lambda(uv)}.$$
    Since $\delta(\lambda_{1}\lambda_{2})=(\delta\lambda_{1})(\delta\lambda_{2})$, we have that $B^{2}(T,\F^{\times})=\{\delta\lambda\mid \lambda\colon T\to\F^{\times} \}$ is a subgroup of $Z^{2}(T,\F^{\times})$. We denote the quotient by $H^{2}(T,\F^{\times})=Z^{2}(T,\F^{\times})/B^{2}(T,\F^{\times})$ and call it the \emph{second cohomology group of $T$}. Given $\sigma \in Z^{2}(T,\F^{\times})$, we denote by $[\sigma]$ the class of $\sigma$ in $H^{2}(T,\F^{\times})$. It is known that $\F^{\sigma_{1}}T\cong\F^{\sigma_{2}}T$ if and only if $[\sigma_{1}]=[\sigma_{2}]$.
    
    \subsection{Graded-division algebras}
    Assume that $\F$ is algebraically closed and $\cd=\bigoplus _{g\in G}\cd_{g}$ is a finite-dimensional graded-division algebra. Let $T=\supp\, \cd$ be its support. Then, $T$ is a subgroup of $G$. Moreover, $\cd_{1}\supseteq \F$. Since $\F$ is algebraically closed and $\cd_{1}$ is finite-dimensional, we conclude that $\cd_{1}=\F$. This also implies that $\dim\cd_{t}=1$ for all $t \in T$. Let $\{X_{t}\mid t\in T\}$ be a basis of $\cd$ where each $X_{t}\in \cd_{t}$. Then, there exists a map $\sigma\colon T\times T\to \F$ such that $X_{u}X_{v}=\sigma(u,v)X_{uv}$ for all $u,v \in T$. Since $\cd$ is associative, we conclude that $\sigma$ has to be a $2$-cocycle. Hence, $\cd\cong \F^{\sigma}T$.

    A map $\beta\colon T\times T\to \F^{\times}$ is called an \emph{alternating bicharacter} if
    \begin{align*}
        \beta(uv,w)&=\beta(u,w)\beta(v,w),&\forall \, u,v,w\in T;\\
        \beta(u,vw)&=\beta(u,v)\beta(u,w),&\forall \, u,v,w\in T;\\
        \beta(u,u)&=1,&\forall \, u\in T.
    \end{align*}
    Moreover, $\beta$ is called \emph{nondegenerate} if $\beta(u,t)=1=\beta(t,u)$ for all $u \in T$ implies $t = 1$. Now, assume that $T$ is abelian. For any $\sigma \in Z^{2}(T,\F^{\times})$, we obtain an alternating bicharacter $\beta=\beta_{\sigma}$ given by
    $$\beta_{\sigma}(u,v)=\frac{\sigma(u,v)}{\sigma(v,u)}.$$
    One can prove that $\beta_{\sigma}$ only depends on the cohomology class $[\sigma]\in H^{2}(T,\F^{\times})$. Furthermore, $\mathcal{D}=\F^{\sigma}T$ is central if and only if $\beta=\beta_{\sigma}$ is nondegenerate. Finally, the pair $(T,\beta)$ uniquely determines an isomorphism class of finite-dimensional central graded-division algebras over $\F$ with commutative support (see \cite[Theorem 2.15]{EK2013}).

    Assume that $\cd=\F^{\sigma}T$ is a finite-dimensional central graded-division algebra over $\F$ with commutative support determined by the pair $(T,\beta)$. Then, we can decompose $T$ as a direct product of cyclic groups
    $$T=H_{1}'\times H_{1}''\times\cdots\times H_{r}'\times H_{r}''$$
    such that $H_{i}'\times H_{i}''$ and $H_{j}'\times H_{j}''$ are $\beta$-orthogonal for $i\neq j$, and $H_{i}'$ and $H_{i}''$ are in duality by $\beta$. Then, $H_{i}'\cong H_{i}''\cong \mathbb{Z}_{\ell_{i}}$ for some $\ell_{i}\in \mathbb{N}$ and $\ell_{1}\cdots\ell_{r}=\sqrt{|T|}$. Moreover, for each $i=1,\ldots,r$ we can find a generator $a_{i}$ of $H_{i}'$ and a generator $b_{i}$ of $H_{i}''$ such that $\beta(a_{i},b_{i})=\varepsilon_{i}$, where $\varepsilon_{i}$ is a primitive $\ell_{i}$-th root of unity. Then, we obtain an isomorphism between $\F^{\sigma}T$ and $M_{\ell_{1}}(\F)\otimes\cdots\otimes M_{\ell_{r}}(\F)\cong M_{n}(\F)$ defined by 
    $$X_{a_{i}}\mapsto I\otimes\cdots \otimes X_{i}\otimes\cdots\otimes I\quad\text{and}\quad X_{b_{i}}\mapsto I\otimes\cdots \otimes Y_{i}\otimes\cdots\otimes I,$$
    where $n=\sqrt{|T|}$,
    \begin{align*}
        X_{i}= 
        \begin{pmatrix} 
            \varepsilon_{i}^{\ell_{i}-1} & 0 & \dots & 0 & 0\\
            0 & \varepsilon_{i}^{\ell_{i}-2} & \dots & 0 & 0\\
            \vdots & \vdots & \ddots & \vdots & \vdots\\
            0 & 0 & \dots & \varepsilon_{i} & 0\\
            0 & 0 & \dots & 0 & 1
        \end{pmatrix}\quad\text{and}
        \quad Y_{i} =
        \begin{pmatrix} 
            0 & 1 & 0 & \dots & 0\\
            0 & 0 & 1 & \dots & 0\\
            \vdots & \vdots & \vdots & \ddots & \vdots\\
            0 & 0 & 0 & \dots & 1\\
            1 & 0 & 0 & \dots & 0
        \end{pmatrix}.
    \end{align*}
    This means that we can realize $\cd$ as a matrix algebra. Moreover, we obtain a representative $\sigma'$ of the cohomology class $[\sigma]$ that is multiplicative in each variable. Such a representative is given by
    \begin{equation}
        \label{sigmaRepresentant}
        \sigma'(a_{1}^{i_{1}}b_{1}^{j_{1}}\cdots a_{r}^{i_{r}}b_{r}^{j_{r}},a_{1}^{i_{1}'}b_{1}^{j_{1}'}\cdots a_{r}^{i_{r}'}b_{r}^{j_{r}'})=\varepsilon_{1}^{-j_{1}i_{1}'}\cdots \varepsilon_{r}^{-j_{r}i_{r}'}.
    \end{equation}

    \section{Homogeneous involutions on graded-division algebras}
    Let $\F$ be an algebraically closed field and $T=\prod_{i=1}^{n} (\mathbb{Z}_{\ell_{i}})^{2s_{i}}$ a product of cyclic groups where each $\ell_{i}$ is a power of a prime. Since $T$ is abelian, we shall use the additive notation for $T$ in this section. Let $\cd$ be a central finite-dimensional $T$-graded-division algebra. Then, we can realize $\cd$ as a matrix algebra $M_{r}(\F)$. Also, there exist $\sigma\in Z^{2}(T,\F^{\times})$ such that $\mathcal{D}=\F^{\sigma}T$. 
    Take $a_{ij},b_{ij}\in T$ such that each $a_{ij}$ and $b_{ij}$ has order $\ell_{i}$ and  $T=\prod_{i=1}^{n}\prod_{j=1}^{s_{i}}H_{ij}$ where $H_{ij}=\langle a_{ij},b_{ij}\rangle$. Then, $\{a_{ij}, b_{ij} \}$ generate $T$ and, without loss of generality, we may assume that $\sigma$ has the form \eqref{sigmaRepresentant} with $\varepsilon_{i}=\beta(a_{ij},b_{ij})=\beta(a_{ik},b_{ik})$ for all $(i, j, k)$. 
    Then, we can write $\cd = \F\langle X_{a_{ij}},X_{b_{ij}}\rangle/I$, where $I$ is the ideal of $\F\langle X_{a_{ij}},X_{b_{ij}}\rangle$ generated by the following elements
    \begin{enumerate}
        \item $X_{c_{ij}}^{\ell_{i}}-1$ with $c_{ij} \in \{a_{ij},b_{ij}\}$, $i \in \{1,\ldots, n\}$ and $j\in \{1,\ldots,s_{i}\}$;
        \item $X_{a_{ij}}X_{b_{k\ell}}-\varepsilon_{i}^{\delta_{ik}\delta_{j\ell}}X_{b_{k\ell}}X_{a_{ij}}$ with  $i,k \in \{1,\ldots, n\}$, $j\in \{1,\ldots,s_{i}\}$ and $\ell \in \{1,\ldots,s_{k}\}$;
        \item $X_{a_{ij}}X_{a_{k\ell}}-X_{a_{k\ell}}X_{a_{ij}}$ with $i,k \in \{1,\ldots, n\}$, $j\in \{1,\ldots,s_{i}\}$ and $\ell\in\{1,\ldots,s_{k}\}$;
        \item $X_{b_{ij}}X_{b_{k\ell}}-X_{b_{k\ell}}X_{b_{ij}}$ with $i,k \in \{1,\ldots, n\}$, $j\in \{1,\ldots,s_{i}\}$ and $\ell\in\{1,\ldots,s_{k}\}$.
    \end{enumerate}

    Let $\tau\colon T\to T$. By \Cref{lemmaTau}, if $\cd$ admits a $\tau$-homogeneous anti-automorphism, then $\tau$ must be an anti-automorphism of $T$. Since $T$ is abelian, then $\tau \in \operatorname{Aut}(T)$. Write $T=\prod_{j=1}^{k}T_{p_{j}}$, where each $p_{j}$ is a prime number and $T_{p_{j}}$ is a $p_{j}$-group. Then, $\operatorname{Aut}(T)\cong \prod_{j=1}^{k}\operatorname{Aut}(T_{p_{j}})$. Hence, it is enough to study the case where $T = \prod_{i=1}^{n}(\mathbb{Z}_{p^{i}})^{2s_{i}}$ is a $p$-group for a fixed prime number $p$. Let $\varepsilon=\varepsilon_{n}$. Then, we can assume that $\varepsilon_{i}=\varepsilon^{p^{n-i}}$ for all $i\in\{1,\ldots,n\}$.

    We know that any map $f\colon \{ X_{a_{ij}}, X_{b_{ij}}\mid 1\leq i \leq n,\, 1\leq j\leq s_{i} \}\to\cd$ can be extended to an algebra homomorphism 
    $$f_{1}\colon \F\langle X_{a_{ij}},X_{b_{ij}}\mid 1\leq i \leq n,\, 1\leq j\leq s_{i} \rangle\to \cd$$
    or an algebra anti-homomorphism
    $$f_{-1}\colon \F\langle X_{a_{ij}},X_{b_{ij}}\mid 1\leq i \leq n,\, 1\leq j\leq s_{i} \rangle\to \cd.$$
    Furthermore, $f_{s}$, $s\in\{-1,1\}$, factors through $\cd$ if, and only if
    \begin{align*}
        f_{s}(X_{a_{ij}})^{p^{i}}&=f_{s}(X_{b_{ij}})^{p^{i}}=1;\\
        f_{s}(X_{c})f_{s}(X_{d})&= \beta_{s}(c,d)f_{s}(X_{d})f_{s}(X_{c})
    \end{align*}
    with $\beta_{s}(c,d)=\beta(c,d)^{s}$ for all $c,d\in\{a_{11},b_{11},\ldots, a_{ns_{n}},b_{ns_{n}}\}$. 
    
    Let $\tau \in \operatorname{Aut}(T)$ and denote
    \begin{equation}
        \label{tauGenerators}
        \begin{split}
            \tau(a_{ij})&=\sum\limits_{k=1}^{n}\sum\limits_{\ell=1}^{s_{k}} (\mu_{k\ell}^{(a_{ij})}a_{k\ell} + \nu_{k\ell}^{(a_{ij})}b_{k\ell});\\
            \tau(b_{ij})&= \sum\limits_{k=1}^{n}\sum\limits_{\ell=1}^{s_{k}} (\mu_{k\ell}^{(b_{ij})}a_{k\ell} + \nu_{k\ell}^{(b_{ij})}b_{k\ell})
        \end{split}
    \end{equation}
    with each $\mu_{k\ell}^{(a_{ij})},\mu_{k\ell}^{(b_{ij})},\nu_{k\ell}^{(a_{ij})},\nu_{k\ell}^{(b_{ij})}\in \mathbb{Z}$.  For all $c,d\in\{a_{11},b_{11},\ldots, a_{ns_{n}},b_{ns_{n}}\}$, set
    \begin{align*}
        \tau_{k\ell}^{(c),(d)} &= 
        \begin{pmatrix} 
            \mu_{k\ell}^{(c)} & \mu_{k\ell}^{(d)}\\
            \nu_{k\ell}^{(c)} & \nu_{k\ell}^{(d)}
        \end{pmatrix}.
    \end{align*}
    Now, for a given $\tau\in\operatorname{Aut}(T)$, we fix matrices $\tau_{k\ell}^{(c),(d)}$ whose entries satisfy \eqref{tauGenerators} and define
    $$P_{c,d}=\sum_{k=1}^{n}\sum_{\ell=1}^{s_{k}}p^{n-k}\det\,\tau_{k\ell}^{(c),(d)}.$$

    Note that $P_{c,d}$ does not depends on the choice of $\tau_{k\ell}^{(c),(d)}$, modulo $p^{n}\mathbb{Z}$. In fact, if we have another matrix $\tilde{\tau}_{k\ell}^{(c),(d)}$ whose entries satisfy \eqref{tauGenerators}, then $\tau_{k\ell}^{(c),(d)}-\tilde{\tau}_{k\ell}^{(c),(d)}\in M_{2}(p^{k}\mathbb{Z})$. Then, in particular, $p^{n-k}(\det\,\tau_{k\ell}^{(c),(d)}-\det\,\tilde{\tau}_{k\ell}^{(c),(d)}) \in p^{n}\mathbb{Z}$ for all $c,d\in\{a_{11},b_{11},\ldots, a_{ns_{n}},b_{ns_{n}}\}$. 
    
    \begin{proposition}
        \label{propAntiautomorphismEquivalence}
        Let $\tau \in \operatorname{Aut}(T)$ and let $\lambda_{a_{ij}},\lambda_{b_{ij}}\in \F^{\times}$ for $i\in \{1,\ldots,n\},\,j\in\{1,\ldots,s_{i}\}$. Then,
        \begin{align*}
            \psi\colon\cd &\to \cd\\
            X_{a_{ij}}&\mapsto \lambda_{a_{ij}}X_{\tau(a_{ij})}\\
            X_{b_{ij}}&\mapsto \lambda_{b_{ij}}X_{\tau(b_{ij})}
        \end{align*}
        defines a $\tau$-homogeneous anti-automorphism on $\cd$ if and only if $\lambda_{a_{ij}}^{p^{i}}=\sigma(\tau(a_{ij}),\tau(a_{ij}))^{-\frac{{p^{i}(p^{i}-1)}}{2}}$, $\lambda_{b_{ij}}^{p^{i}}=\sigma(\tau(b_{ij}),\tau(b_{ij}))^{-\frac{{p^{i}(p^{i}-1)}}{2}}$ and  the following congruence modulo $p^{n}\mathbb{Z}$ holds
        \begin{equation}
            \label{equationPCongruentModpnZ}
            P_{c,d}\equiv\begin{cases}
                -p^{n-i},&\text{if $c=a_{ij}$ and $d=b_{ij}$}\\
                p^{n-i},&\text{if $c=b_{ij}$ and $d=a_{ij}$}\\
                0, &\text{otherwise}
            \end{cases}
        \end{equation}
        for all $c,d\in\{a_{11},b_{11},\ldots, a_{ns_{n}},b_{ns_{n}}\}$. Moreover, if $\psi$ is a homogeneous anti-automorphism, we have $\lambda_{a_{ij}}^{p^{i}}=\lambda_{b_{ij}}^{p^{i}}=1$ if $p \neq 2$ and
        $\lambda_{a_{ij}}^{p^{i}},\lambda_{b_{ij}}^{p^{i}}\in\{-1,1\}$ if $p=2$.
    \end{proposition}
    \begin{proof}
        Let $c_{ij}\in \{a_{ij},b_{ij}\}$. Since
        \begin{align*}
            \psi(X_{c_{ij}})^{p^{i}}=\lambda_{c_{ij}}^{p^{i}}X_{\tau(c_{ij})}^{p^{i}}=\lambda_{c_{ij}}^{p^{i}}\sigma(\tau(c_{ij}),\tau(c_{ij}))^{\frac{{p^{i}(p^{i}-1)}}{2}}X_{0},
        \end{align*}
        we have that $\psi(X_{c_{ij}})^{p^{i}} = 1$ is equivalent to $\lambda_{c_{ij}}^{p^{i}}\sigma(\tau(c_{ij}),\tau(c_{ij}))^{\frac{{p^{i}(p^{i}-1)}}{2}}=1$.
        Moreover,
        \begin{align*}
            &\psi(X_{d})\psi(X_{c})-\beta(c,d)\psi(X_{c})\psi(X_{d})\\
            &=\lambda_{c}\lambda_{d}(X_{\tau(d)}X_{\tau(c)}-\beta(c,d) X_{\tau(c)}X_{\tau(d)})\\
            &=\lambda_{c}\lambda_{d}(\beta(\tau(d),\tau(c))-\beta(c,d))X_{\tau(c)}X_{\tau(d)}.
        \end{align*}
        Then, $\psi(X_{d})\psi(X_{c})=\beta(c,d)\psi(X_{c})\psi(X_{d})$ is equivalent to $\beta(c,d)=\beta(\tau(d),\tau(c))=\beta(\tau(c),\tau(d))^{-1}$.
        Now, calculating $\beta(\tau(c),\tau(d))$, we have
        \begin{align*}
            &\beta (\tau(c),\tau(d))=\\
            &=\beta \left(\sum\limits_{k=1}^{n}\sum\limits_{\ell=1}^{s_{k}} (\mu_{k\ell}^{(c)}a_{k\ell} + \nu_{k\ell}^{(c)}b_{k\ell}),\sum\limits_{k=1}^{n}\sum\limits_{\ell=1}^{s_{k}} (\mu_{k\ell}^{(d)}a_{k\ell} + \nu_{k\ell}^{(d)}b_{k\ell})\right)\\
            &=\prod\limits_{k=1}^{n}\prod\limits_{\ell=1}^{s_{k}}\beta(a_{k\ell},b_{k\ell})^{\mu_{k\ell}^{(c)}\nu_{k\ell}^{(d)}}\beta(b_{k\ell},a_{k\ell})^{\nu_{k\ell}^{(c)}\mu_{k\ell}^{(d)}}\\
            &=\prod\limits_{k=1}^{n}\prod\limits_{\ell=1}^{s_{k}}\varepsilon_{k}^{\mu_{k\ell}^{(c)}\nu_{k\ell}^{(d)}}\varepsilon_{k}^{-\nu_{k\ell}^{(c)}\mu_{k\ell}^{(d)}}=\prod\limits_{k=1}^{n}\prod\limits_{\ell=1}^{s_{k}}\varepsilon_{k}^{\mu_{k\ell}^{(c)}\nu_{k\ell}^{(d)}-\nu_{k\ell}^{(c)}\mu_{k\ell}^{(d)}}\\
            &=\prod\limits_{k=1}^{n}\prod\limits_{\ell=1}^{s_{k}}\varepsilon_{k}^{\det(\tau_{k\ell}^{(c),(d)})}=\prod\limits_{k=1}^{n}\prod\limits_{\ell=1}^{s_{k}}\varepsilon^{p^{n-k}\det(\tau_{k\ell}^{(c),(d)})}=\varepsilon^{P_{c,d}}.
        \end{align*}
        Since
        \begin{align*}
            \beta(c,d)=\begin{cases}
                \varepsilon^{p^{n-i}},&\text{ if $c=a_{ij}$ and $d=b_{ij}$}\\
                \varepsilon^{-p^{n-i}},&\text{ if $c=b_{ij}$ and $d=a_{ij}$}\\
                1, &\text{ otherwise}
            \end{cases}
        \end{align*}
        and $\varepsilon$ is a primitive $p^{n}$-th root of unity, we have that $\psi(X_{d})\psi(X_{c})=\beta(c,d)\psi(X_{c})\psi(X_{d})$ is equivalent to \eqref{equationPCongruentModpnZ}.
    \end{proof}
    
    \begin{example}
        Notice that in the previous proposition it can occur $\lambda_{a_{ij}}^{p^{i}}\neq\lambda_{b_{ij}}^{p^{i}}$. In fact, take $T=\mathbb{Z}_{2}^{2}=\langle a,b\rangle$. Then, the map $\tau$ that sends $a$ to $a+b$ and $b$ to $b$ is an automorphism of $T$. Notice that $\sigma(\tau(a),\tau(a))=-1$ and $\sigma(\tau(b),\tau(b))=1$. Then, the map that sends $X_{a}$ to $iX_{a+b}$ and $X_{b}$ to $X_{b}$ where $i^{2}=-1$ defines a $\tau$-homogeneous anti-automorphism on $\cd$.
    \end{example}
    
    \begin{remark}
        \label{remarkAutomorphismEquivalence}
        By replacing the condition \eqref{equationPCongruentModpnZ} in \Cref{propAntiautomorphismEquivalence} with
        $$
            P_{c,d}\equiv
            \begin{cases}
                p^{n-i},&\text{if $c=a_{ij}$ and $d=b_{ij}$}\\
                -p^{n-i},&\text{if $c=b_{ij}$ and $d=a_{ij}$}\\
                0, &\text{otherwise}    
            \end{cases}
        $$
        we obtain the description of $\tau$-homogeneous automorphisms.
    \end{remark}
    
    We shall denote the map defined in \Cref{propAntiautomorphismEquivalence} by $(\tau,\lambda_{a_{ij}},\lambda_{b_{ij}})$. 
    
    \begin{remark}
        \label{remtg}
        Take $g = \sum\limits_{i=1}^{n}\sum\limits_{j=1}^{s_{i}}(\alpha_{ij}a_{ij}+\beta_{ij}b_{ij})\in T$. If we write
        $$\psi(X_{g})=\lambda_{g}X_{\tau(g)},$$ 
        then using \eqref{sigmaRepresentant} and \eqref{tauGenerators} we can find $t_{g}\in\{0,\ldots,p^{n}-1\}$ such that 
        $$\lambda_{g} = \varepsilon^{t_{g}}\prod\limits_{i=1}^{n}\prod\limits_{j=1}^{s_{i}}\lambda_{a_{ij}}^{\alpha_{ij}}\lambda_{b_{ij}}^{\beta_{ij}}.$$ 
    \end{remark}
    
    \begin{remark}
        \label{remarkMoreGenralCase}
        Consider the general case where $T=\prod_{i=1}^{n} T_{i}$ is a product of cyclic groups where each $T_{i}$ is a $p_{i}$-group, $p_{1},\ldots,p_{n}$ being distinct prime numbers. Given $\tau\in \operatorname{Aut}(T)$ and $\psi\colon\mathbb{F}^{\sigma}T\to\mathbb{F}^{\sigma}T$ an anti-automorphism we have that $\psi$ is $\tau$-homogeneous if and only if $\psi_{i}$ is $\tau_{i}$-homogeneous for all $i\in\{1,\ldots,n\}$, where $\psi_{i}$ is the restriction of $\psi$ to $\F^{\sigma}T_{i}$ and $\tau_{i}$ is the restriction of $\tau$ to $T_{i}$. 
    \end{remark}
    
    Now, if we consider graded or degree-inverting anti-automorphisms, we have that each $\tau_{k\ell}^{(a_{ij},b_{ij})}=\pm \delta_{ik}\delta_{j\ell} I_{2}$. Therefore, $\det\,\tau_{k\ell}^{(ij)}=\delta_{ik}\delta_{j\ell}$ and then
    $$-p^{n-i}\equiv\sum\limits_{k=1}^{n}\sum_{\ell=1}^{s_{k}}p^{n-k}\det\,\tau_{k\ell}^{(a_{ij},b_{ij})}\equiv p^{n-i}\quad(\text{mod }p^{n}\mathbb{Z})$$
    for all $i=\{1,\ldots,n\}$. It is only possible if $p=2$ and $n=1$. Therefore, we recover the following known result.
    
    \begin{corollary}[{\cite[Lemma 3.8]{FSY22} and \cite[Lemma 2.50]{EK2013}}]
        Let $\mathcal{D}$ be a central finite-dimensional graded-division algebra over an algebraically closed field. If $\cd$ admits a degree-preserving or degree-inverting anti-automorphism and $T=\supp\,\cd$ is abelian, then $T$ is an elementary $2$-group.\qed
    \end{corollary}
    
    If we deal with homogeneous involutions, we have the following straightforward proposition. 
    
    \begin{proposition}
        \label{propInvolution}
        Let $(\tau,\lambda_{a_{ij}},\lambda_{b_{ij}})$ be a homogeneous anti-automorphism. Then, it is an involution if and only if $\tau^{2}=1$ and $\lambda_{a_{ij}}\lambda_{\tau(a_{ij})}=\lambda_{b_{ij}}\lambda_{\tau(b_{ij})}=1$ for all $i\in\{1,\ldots,n\}$, $j\in \{1,\ldots, s_{i}\}$.\qed
    \end{proposition}
    
    Now we shall compare different homogeneous anti-automorphisms.
    
    \begin{proposition}
        \label{propositionEquivalenceAntiautomorphisms}
        Let $\psi = (\tau,\lambda_{a_{ij}},\lambda_{b_{ij}})$ and $\psi' = (\tau',\lambda_{a_{ij}}',\lambda_{b_{ij}}')$ be two homogeneous anti-automorphisms. Then, $\psi$ and $\psi'$ are equivalent if and only if there exist $\phi\in\operatorname{Aut}(T)$ and a map $\chi\colon G\to \F^{\times}$ such that $\frac{\chi(g+h)}{\chi(g)\chi(h)}=\frac{\sigma(\phi(g),\phi(h))}{\sigma(g,h)}$ for all $g,h\in T$, $\tau'=\phi\tau\phi^{-1}$ and        $$\lambda_{c}=\chi(c)\chi(\tau(c))^{-1}\lambda_{\phi(c)}'$$
        for all $c\in\{a_{11},b_{11},\ldots, a_{ns_{n}},b_{ns_{n}}\}$. In this case, the equivalence is given by $\varphi\colon X_{g}\in\mathcal{D}\mapsto\chi(g)X_{\phi(g)}\in\mathcal{D}$.
    \end{proposition}
    \begin{proof}
        Suppose that $\psi$ and $\psi'$ are equivalent. Then, there exists a $\phi$-homogeneous automorphism $\varphi\colon X_{g}\in\mathcal{D}\mapsto\chi(g)X_{\phi(g)}\in\mathcal{D}$ for some $\phi\in \operatorname{Aut}(T)$, with $\chi (g) \in \F ^\times $, such that $\varphi\circ \psi=\psi'\circ\varphi$. It implies that, for all $g\in T$, we have
        \begin{align*}
        \varphi\circ\psi(X_{g})&=\varphi(\lambda_{g}X_{\tau(g)})=\chi({\tau(g)})\lambda_{g}X_{\phi\tau(g)};\\
            \psi'\circ\varphi(X_{g})&= \psi'(\chi(g)X_{\phi(g)})=\chi(g)\lambda_{\phi(g)}'X_{\tau'\phi(g)}.
        \end{align*}
        Therefore, $\tau'=\phi\tau\phi^{-1}$ and $\lambda_{g}=\chi(g)\chi(\tau(g))^{-1}\lambda_{\phi(g)}'$. Moreover,
        \begin{align*}
            \chi(g)\chi(h)X_{\phi(g)}X_{\phi(h)}&=\varphi(X_{g})\varphi(X_{h})=\varphi(X_{g}X_{h})=\sigma(g,h)\varphi(X_{g+h})\\
            &=\sigma(g,h)\chi(g+h)X_{\phi(g+h)}\\
            &=\frac{\sigma(g,h)}{\sigma(\phi(g),\phi(h))}\chi(g+h)X_{\phi(g)}X_{\phi(h)}
        \end{align*}
        for all $g,h\in G$. Therefore, $\frac{\chi(g+h)}{\chi(g)\chi(h)}=\frac{\sigma(\phi(g),\phi(h))}{\sigma(g,h)}$.
    
        Conversely, assume that there exists $\phi\in\operatorname{Aut}(T)$ and a map $\chi\colon G\to \F^{\times}$ such that $\frac{\chi(g+h)}{\chi(g)\chi(h)}=\frac{\sigma(\phi(g),\phi(h))}{\sigma(g,h)}$ for all $g,h\in T$, $\tau'=\phi\tau\phi^{-1}$ and $\lambda_{c}=\chi(c)\chi(\tau(c))^{-1}\lambda_{\phi(c)}'$ for all $c\in\{a_{11},b_{11},\ldots, a_{ns_{n}},b_{ns_{n}}\}$. Then,
        \begin{align*}
            \varphi\colon \cd &\to \cd\\
            X_{g}&\mapsto\chi(g)X_{\phi(g)}
        \end{align*}
        is an equivalence. In fact,
        \begin{align*}
            \varphi(X_{g}X_{h})&=\sigma(g,h)\varphi(X_{g+h})=\sigma(g,h)\chi(g+h)X_{\phi(g+h)}\\&=\frac{\sigma(g,h)}{\sigma(\phi(g),\phi(h))}\chi(g+h)X_{\phi(g)}X_{\phi(h)}=\chi(g)\chi(h)X_{\phi(g)}X_{\phi(h)}\\
            &=\varphi(X_{g})\varphi(X_{h})
        \end{align*}
        for all $g, h \in G$.
        Moreover,
        \begin{align*}
            \varphi\circ\psi(X_{c})= \chi(\tau(c))\lambda_{c}X_{\phi\tau(c)}=\chi(c)\lambda_{\phi(c)}'X_{\tau'\phi(c)}=\psi'\circ\varphi(X_{c})
        \end{align*}
        for all $c\in\{a_{11},b_{11},\ldots, a_{ns_{n}},b_{ns_{n}}\}$. Therefore, $\psi$ and $\psi'$ are equivalent.
    \end{proof}
    
    In particular, we have the following corollary.
    
    \begin{corollary}
        \label{propIsomorphismAntiAuto}
        Let $\psi = (\tau,\lambda_{a_{ij}},\lambda_{b_{ij}})$ and $\psi' = (\tau',\lambda_{a_{ij}}',\lambda_{b_{ij}}')$ be two homogeneous anti-automorphisms. Then, $\psi$ and $\psi'$ are isomorphic if and only if $\tau = \tau'$ and there exists $\chi\in \hat{T}$ character of $T$ such that
        $$\chi(\tau(c))\lambda_{c}'=\chi(c)\lambda_{c}$$
        for all $c\in\{a_{11},b_{11},\ldots, a_{ns_{n}},b_{ns_{n}}\}$.\qed
    \end{corollary}
    
    \subsection{Pauli grading}
    Now we will give some attention to the case where the grading group is $T = \mathbb{Z}_{n}^{2}$. Then,
    $$\cd = \F\langle X_{a}, X_{b}\mid X_{a}^{n}=X_{b}^{n}=1,\, X_{a}X_{b}=\varepsilon X_{b}X_{a}\rangle$$
    where $\varepsilon$ is a primitive $n$-th root of unity. In this case we know that $\operatorname{Aut}(T)=\operatorname{GL}_{2}(\mathbb{Z}_{n})$. Then, we have the following particularizations of our previous results.
    
    \begin{corollary}
        \label{corollaryAntiautomorphism}
        Let $\tau \in \operatorname{GL}_{2}(\mathbb{Z}_{n})$ and let $\lambda_{a},\lambda_{b}\in \F^{\times}$. Then,    
        \begin{align*}
            \psi\colon\cd &\to \cd\\
            X_{a}&\mapsto \lambda_{a}X_{\tau(a)}\\
            X_{b}&\mapsto \lambda_{b}X_{\tau(b)}
        \end{align*}
        defines a $\tau$-homogeneous anti-automorphism on $\cd$ if and only if we have $\det\tau=-1$, $\lambda_{a
        }^{n}=\sigma(\tau(a),\tau(a))^{-\frac{n(n-1)}{2}}$ and $\lambda_{b}^{n}=\sigma(\tau(b),\tau(b))^{-\frac{n(n-1)}{2}}$. Moreover, if $\psi$ is a homogeneous anti-automorphism, we have $\lambda_{a}^{n}=\lambda_{b}^{n}=1$ if $n$ is odd and
        $\lambda_{a}^{n},\lambda_{b}^{n}\in\{-1,1\}$ if $n$ is even.\qed
    \end{corollary}
    
    \begin{remark}
        \label{remarkAutomorphismEquivalenceParticularCase}
        By \Cref{remarkAutomorphismEquivalence}, the conditions of \Cref{corollaryAntiautomorphism}, replacing $\det\tau=-1$ with
        $\tau\in\operatorname{SL}_{2}(\mathbb{Z}_{n})$
        describe $\tau$-homogeneous automorphisms in this case. 
    \end{remark} 
    
    \begin{corollary}
        \label{corollaryInvolution}
        Let $(\tau,\lambda_{a},\lambda_{b})$ be a homogeneous anti-automorphism. Then, it is an involution if and only if $\operatorname{tr}(\tau)=0$ and $\lambda_{a}\lambda_{\tau(a)}=\lambda_{b}\lambda_{\tau(b)}=1$ for all $i\in\{1,\ldots,n\}$, $j\in \{1,\ldots, s_{i}\}$.
    \end{corollary}
    \begin{proof}
        It is enough to note that if $\det\tau=-1$, then $\tau^{2}=1$ if and only if $\operatorname{tr}(\tau)=0$. 
    \end{proof}
    
    Our next goal is to show that we always have a $\tau$-homogeneous involution for all $\tau\in\operatorname{GL}_{2}(\mathbb{Z}_{n})$ satisfying $\det\tau=-1$ and $\operatorname{tr}(\tau)=0$. Before going to it, remember that a quasi-regular ideal $J$ of a unital ring $\mathcal{A}$ is an ideal of $\mathcal{A}$ in which $1-x$ is invertible for any $x \in J$. For example, if $\mathcal{A}=\mathbb{Z}_{p^{k}}$ with $p$ being a prime number and $k>1$, then any ideal of the form $J=p^{i}\mathbb{Z}_{p^{k}}$ is a quasi-regular ideal of $\mathcal{A}$. The following lemma is well-known. 
    

    
    \begin{lemma}
        \label{lemmaInvertibleQuasiregular}
        Let $\mathcal{A}$ be a unital ring and $J\subseteq \mathcal{A}$ a quasi-regular ideal. If $x\in\mathcal{A}/J$ is invertible, then for any $a \in \mathcal{A}$ such that $a+J=x$, we have that $a$ is invertible and $a^{-1}+J=x^{-1}$.
    \end{lemma}
    \begin{proof}
        Let $a,b\in\mathcal{A}$ such that $a+J=x$ and $b+J=x^{-1}$. Then, $ab=1-y$ for some $y\in J$. Since $J$ is quasi-regular, $1-y$ is invertible. Thus, $a$ is invertible.
    \end{proof}
    
        If $C$ is a commutative unital ring, then the group $\operatorname{SL}_{2}(C)$ acts on $\operatorname{GL}_{2}(C)$ by conjugation. Take $\tau\in\operatorname{GL}_{2}(\mathbb{Z}_{n})$ with $\det\tau=-1$ and $\operatorname{tr}(\tau)=0$. \Cref{propositionEquivalenceAntiautomorphisms} and \Cref{remarkAutomorphismEquivalenceParticularCase} show that there exists a $\tau$-homogeneous involution on $\cd$ if and only if there exists a $\tau'$-homogeneous involution on $\cd$ for any $\tau'$ in the orbit of $\tau$. With this in mind, we shall investigate the orbits of the elements $\tau\in\operatorname{GL}_{2}(\mathbb{Z}_{n})$ with $\det\tau=-1$ and $\operatorname{tr}(\tau)=0$. Let $A,B\in \operatorname{GL}_{2}(C)$. We shall denote that $A$ and $B$ are in the same $\mathrm{SL}_2(C)$-orbit by $A\sim B$.

    For each $n\in\mathbb{N}$, we define the set $\mathcal{O}_{n}$ as the subset of $\operatorname{GL}_{2}(\mathbb{Z}_{n})$ given by:
    \begin{itemize}
        \item 
        $\begin{pmatrix}
            1 & 0\\
            0 & -1
        \end{pmatrix}\in\mathcal{O}_{n}$;
        \item if $n$ is even, then 
        $\begin{pmatrix}
            0 & 1\\
            1 & 0
        \end{pmatrix}\in\mathcal{O}_{n}$;
        \item if $4\mid n$, then
        $\begin{pmatrix}
            1           & 2\\
            \frac{n}{2} & -1
        \end{pmatrix}\in\mathcal{O}_{n}$.
    \end{itemize}

    \begin{lemma}
        \label{lemmaEquivalenceClassMatrices}
        Let $\tau\in\operatorname{GL}_{2}(\mathbb{Z}_{n})$ such that $\det\tau=-1$ and $\operatorname{tr}(\tau)=0$. Then, $\tau\sim\theta$ for some $\theta \in \mathcal{O}_{n}$.
    \end{lemma}
    \begin{proof}
        We will first show the result for $n=p^{i}$ where $p$ is a prime number and $i>0$. We will proceed by induction on $i$. The base case follows from the fact that $\mathbb{Z}_{p}$ is a field. Now, let $i>1$. Suppose the result is valid for $i-1$. Then, consider $\tau'$ being the image of $\tau$ in $M_{2}(\mathbb{Z}_{n})/M_{2}(p^{i-1}\mathbb{Z}_{n})\cong M_{2}(\mathbb{Z}_{p^{i-1}})$. Then, $\det\tau'=-1$ and $\operatorname{tr}(\tau')=0$. 
        
        If $p\neq 2$, by induction hypothesis and \Cref{lemmaInvertibleQuasiregular}, we can assume that $\tau'$ has the form
        $$\tau'=
        \begin{pmatrix}
            1            & q_{1}p^{i-1}\\
            q_{2}p^{i-1} & -1
        \end{pmatrix}.$$
        Since $n$ is odd, $2\in\mathbb{Z}_{n}$ is invertible. Let $P = 
        \begin{pmatrix}
                 1             & -2^{-1}q_{1}p^{i-1}\\
            2^{-1}q_{2}p^{i-1} &         1
        \end{pmatrix}\in M_{2}(\mathbb{Z}_{n})$. Notice that $\det P=1$ and 
        $$P^{-1}\tau' P= 
        \begin{pmatrix}
            1 & 0\\
            0 & -1
        \end{pmatrix}.$$

        If $p=2$, by the induction hypothesis and \Cref{lemmaInvertibleQuasiregular}, we can assume that $\tau'$ has one of the following forms
        \begin{itemize}
            \item $\tau'=
            \begin{pmatrix}
                1+2^{i-1}q_{1} & 2^{i-1}q_{2}\\
                2^{i-1}q_{3}   & -1 - 2^{i-1}q_{1}
            \end{pmatrix}$;
            \item $\tau'=
            \begin{pmatrix}
                2^{i-1}q_{1}   & 1 + 2^{i-1}q_{2}\\
                1-2^{i-1}q_{2}& -2^{i-1}q_{1}
            \end{pmatrix}$.
        \end{itemize}
        In the first case, $\tau'$ can be in the same orbit as
        $\begin{pmatrix}
            1 & 0\\
            0 & -1
        \end{pmatrix}$ or
        $\begin{pmatrix}
            1           & 2\\
            2^{i-1} & -1
        \end{pmatrix}$ as the next table shows.
        
        \begin{tabular}{|c | c | c | c |} 
            \hline
            &&&\\[-1em]
            $i$ & $\tau'$ & $P$ & $P^{-1}\tau'P$\\[0.3em]
            \hline
            &&&\\[-1em]
            $\geq 2$ & 
            $\begin{pmatrix}
                1 & 0\\
                0 & -1
            \end{pmatrix}$ & 
            $\begin{pmatrix}
                1 & 0\\
                0 & 1
            \end{pmatrix}$ & 
            $\begin{pmatrix}
                1 & 0\\
                0 & -1
            \end{pmatrix}$\\[1em]
            \hline
            &&&\\[-1em]
            $\geq 2$ & 
            $\begin{pmatrix}
                1 & 2^{i-1}\\
                0 & -1
            \end{pmatrix}$ & 
            $\begin{pmatrix}
                1 & 2^{i-2}\\
                0 & 1
            \end{pmatrix}$ &
            $\begin{pmatrix}
                1 & 0\\
                0 & -1
            \end{pmatrix}$\\[1em]
            \hline
            &&&\\[-1em]
            $\geq 2$ & 
            $\begin{pmatrix}
                1             & 0\\
                2^{i-1} & -1
            \end{pmatrix}$ & 
            $\begin{pmatrix}
                1              & 0\\
                2^{i-2} & 1
            \end{pmatrix}$ &
            $\begin{pmatrix}
                1 & 0\\
                0 & -1
            \end{pmatrix}$\\[1em]
            \hline
            &&&\\[-1em]
            $2$ & 
            $\begin{pmatrix}
                1 & 2\\
                2 & -1
            \end{pmatrix}$ & 
            $\begin{pmatrix}
                1 & 0\\
                0 & 1
            \end{pmatrix}$ & 
            $\begin{pmatrix}
                1 & 2\\
                2 & -1
            \end{pmatrix}$\\[1em]
            \hline
            &&&\\[-1em]
            $\geq 3$ & 
            $\begin{pmatrix}
                1 & 2^{i-1}\\
                2^{i-1} & -1
            \end{pmatrix}$ & 
            $\begin{pmatrix}
                2^{i-2} + 1 & 2^{i-2}\\
                2^{i-2}     & -2^{i-2} + 1
            \end{pmatrix}$ &
            $\begin{pmatrix}
                1 & 0\\
                0 & -1
            \end{pmatrix}$\\[1em]
            \hline
            &&&\\[-1em]
            $2$ & 
            $\begin{pmatrix}
                -1 & 0\\
                0  & 1
            \end{pmatrix}$ &
            $\begin{pmatrix}
                0  & 1\\
                -1 & 0
            \end{pmatrix}$ &
            $\begin{pmatrix}
                1 & 0\\
                0 & -1
            \end{pmatrix}$\\[1em]
            \hline
            &&&\\[-1em]
            $\geq 3$ & 
            $\begin{pmatrix}
                1 + 2^{i-1}  & 0\\
                0            & -1 - 2^{i-1}
            \end{pmatrix}$ &
            $\begin{pmatrix}
                1 + 2^{i-2} & 1\\
                -2^{i-2}    & 1 - 2^{i-1}
            \end{pmatrix}$ &
            $\begin{pmatrix}
                1            & 2\\
                2^{i-1} & -1
            \end{pmatrix}$\\[1em]
            \hline
            &&&\\[-1em]
            $2$ & 
            $\begin{pmatrix}
                -1 & 2\\
                0  & 1
            \end{pmatrix}$  &
            $\begin{pmatrix}
                1  & 1\\
                -1 & 0
            \end{pmatrix}$ &
            $\begin{pmatrix}
                1 & 0\\
                0 & -1
            \end{pmatrix}$\\[1em]
            \hline
            &&&\\[-1em]
            $\geq 3$ & 
            $\begin{pmatrix}
                1 + 2^{i-1} & 2^{i-1}\\
                0           & -1 - 2^{i-1}
            \end{pmatrix}$ & 
            $\begin{pmatrix}
                1              & 1\\
                2^{i-2} & 1 + 2^{i-2}
            \end{pmatrix}$ &
            $\begin{pmatrix}
                1            & 2\\
                2^{i-1} & -1
            \end{pmatrix}$\\[1em]
            \hline
            &&&\\[-1em]
            $2$ & 
            $\begin{pmatrix}
                -1 & 0\\
                2  & 1
            \end{pmatrix}$  &
            $\begin{pmatrix}
                0  & 1\\
                -1 & 1
            \end{pmatrix}$ &
            $\begin{pmatrix}
                1 & 0\\
                0 & -1
            \end{pmatrix}$\\[1em]
            \hline
            &&&\\[-1em]
            $\geq 3$ & 
            $\begin{pmatrix}
                1 + 2^{i-1} & 0\\
                2^{i-1}     & -1 - 2^{i-1}
            \end{pmatrix}$  & 
            $\begin{pmatrix}
                1 & 1 + 2^{i-2}\\
                0 & 1
            \end{pmatrix}$ &
            $\begin{pmatrix}
                1            & 2\\
                2^{i-1} & -1
            \end{pmatrix}$\\[1em]
            \hline
            &&&\\[-1em]
            $\geq 2$ & 
            $\begin{pmatrix}
                1 + 2^{i-1}  & 2^{i-1}\\
                2^{i-1} & -1 - 2^{i-1}
            \end{pmatrix}$ &
            $\begin{pmatrix}
                1 & 1\\
                0 & 1
            \end{pmatrix}$ &
            $\begin{pmatrix}
                1            & 2\\
                2^{i-1} & -1
            \end{pmatrix}$\\[1em]
            \hline
        \end{tabular}

        In the second case, let $P=
        \begin{pmatrix}
            1-q_{2}        & q_{2} + 2q_{1}q_{2}\\
            2q_{1} - q_{2} & 1 - q_{2} + 2q_{1}q_{2}
        \end{pmatrix}$ if $i=2$, $P=
        \begin{pmatrix}
            1               & 2q_{2}\\
            4q_{1} + 2q_{2} & 1 + 4q_{2}
        \end{pmatrix}$ if $i=3$ and $P=
        \begin{pmatrix}
            2^{i-2}q_{2}-1 & 0\\
            2^{i-1}q_{1} & -2^{i-2}q_{2} - 1
        \end{pmatrix}$ if $i>3$. Notice that $\det P = 1$ and $$P^{-1}\tau' P= 
        \begin{pmatrix}
            0 & 1\\
            1 & 0
        \end{pmatrix}.$$

        Then, by induction, the proof of the lemma is complete when $n$ is a power of a prime number. 
        
        Now, let $n\in\mathbb{N}$ and $n=p_{1}^{\alpha_{1}}\cdots p_{k}^{\alpha_{k}}$ be its decomposition in prime numbers. Write $q_{i}=p_{i}^{\alpha_{i}}$. Let $\tau\in\operatorname{GL}_{2}(\mathbb{Z}_{n})$ such that $\det\tau=-1$ and $\operatorname{tr}(\tau)=0$. Let $\pi_{q_{i}}\colon \mathbb{Z}_{n}\to\mathbb{Z}_{q_{i}}$ the canonical surjective homomorphism. Then, $\tau_{q_{i}}=\pi_{q_{i}}(\tau)$ has determinant equal to $-1$ and trace equal to $0$. Hence, $\tau_{q_{i}}\sim \theta_{i}$ for some $\theta_{i}\in\mathcal{O}_{q_{i}}$. Now, using an isomorphism $\mathbb{Z}_{n}\cong\mathbb{Z}_{q_{1}}\times\cdots\times\mathbb{Z}_{q_{k}}$ and the fact that
        $$
        \begin{pmatrix}
            1 & 0\\
            0 & -1
        \end{pmatrix} \sim
        \begin{pmatrix}
            0 & 1\\
            1 & 0
        \end{pmatrix} \sim
        \begin{pmatrix}
            1 & 2\\
            0 & -1
        \end{pmatrix}
        $$
        in $\mathbb{Z}_{q_{i}}$ for all odd $q_{i}$ we conclude that $\tau\sim \theta$ for some $\theta\in\mathcal{O}_{n}$.
    \end{proof}

    \begin{remark}
        If $n$ is even, then 
        $\begin{pmatrix}
            1 & 0\\
            0 & -1
        \end{pmatrix}$ and
        $\begin{pmatrix}
            0 & 1\\
            1 & 0
        \end{pmatrix}$ are not similar. Moreover, if $n\mid 4$, $
        \begin{pmatrix}
            1           & 2\\
            \frac{n}{2} & -1
        \end{pmatrix}$ is not similar to 
        $\begin{pmatrix}
            1 & 0\\
            0 & -1
        \end{pmatrix}$ nor
        $\begin{pmatrix}
            0 & 1\\
            1 & 0
        \end{pmatrix}$.
    \end{remark}

    Let $n\in\mathbb{N}$. For the next theorem, we shall denote the elements of $\mathcal{O}_{n}$ as follows.
    \begin{align*}
        \theta_{1}&=
        \begin{pmatrix}
            1 & 0\\
            0 & -1
        \end{pmatrix},\quad
        \theta_{2}=
        \begin{pmatrix}
            0 & 1\\
            1 & 0
        \end{pmatrix}\quad\text{and}\quad
        \theta_{3}=
        \begin{pmatrix}
            1 & 2\\
            \frac{n}{2} & -1
        \end{pmatrix}.
    \end{align*}

    \begin{theorem}
        Let $T=\mathbb{Z}_{n}^{2}$ and $\F$ be an algebraically closed field of characteristic different from $n$. Let $\cd= M_{n}(\F)$ be a matrix algebra endowed with a division grading with support $T$. If $\tau \in \operatorname{GL}_{2}(\mathbb{Z}_{n})$ with $\det\tau=-1$ and $\operatorname{tr}(\tau)=0$, then there exists a $\tau$-homogeneous involution on $\cd$. Moreover,
        \begin{enumerate}
            \item If $n$ is odd, all homogeneous involutions are equivalent to $(\theta_{1},1,1)$ and there is a unique isomorphism class of $\tau$-homogeneous involutions;
            \item If $n$ is even and $4\nmid n$,
            \begin{enumerate}
                \item any homogeneous involution is equivalent to $(\theta_{1},1,1)$, $(\theta_{1},-1,\varepsilon)$ or $(\theta_{2},1,1)$;
                \item if $\tau\sim\theta_{1}$, there are four isomorphism classes of $\tau$-homogeneous involutions;
                \item if $\tau\sim\theta_{2}$, there is a unique isomorphism class of $\tau$-homogeneous involutions;
            \end{enumerate}
            \item If $4\mid n$,
            \begin{enumerate}
                \item any homogeneous involution is equivalent to $(\theta_{1},1,1)$, $(\theta_{1},1,\varepsilon)$, $(\theta_{1},-1,\varepsilon)$, $(\theta_{2},1,1)$ or $(\theta_{3},\varepsilon^{\frac{n}{4}},\varepsilon)$;
                \item if $\tau\sim\theta_{1}$, there are four isomorphism classes of $\tau$-homogeneous involutions;
                \item if $\tau\sim\theta_{2}$, there is a unique isomorphism class of $\tau$-homogeneous involutions;
                \item if $\tau\sim\theta_{3}$, there is a unique isomorphism class of $\tau$-homogeneous involutions;
            \end{enumerate}
        \end{enumerate}
    \end{theorem}
    \begin{proof}
        As discussed before, by \Cref{lemmaEquivalenceClassMatrices}, it is enough to consider the case where $\tau\in\mathcal{O}_{n}$.
        
        Let $\tau=\theta_{1}$. Then, by \Cref{corollaryAntiautomorphism} and \Cref{corollaryInvolution}, $(\tau,\lambda_{a},\lambda_{b})$ is a $\tau$-homogeneous involution if and only if $\lambda_{a},\lambda_{b}$ are $n$-th roots of unity and $\lambda_{a}^{2}=1$. Let $\psi=(\tau,\lambda_{a},\lambda_{b})$ and $\psi'=(\tau,\lambda_{a}',\lambda_{b}')$ be two $\tau$-homogeneous involutions. Then, $\lambda_{a},\lambda_{a}'\in\{-1,1\}$, $\lambda_{b}= \varepsilon^{\beta}$ and $\lambda'_{b}= \varepsilon^{\beta'}$ for some $\beta,\beta'\geq 0$. By \Cref{propIsomorphismAntiAuto}, $\psi$ and $\psi'$ are isomorphic if and only if $\lambda_{a}=\lambda_{a}'$ and there exists $\chi\in \hat{T}$ such that $\varepsilon^{\beta'-\beta}=\chi(b)^{2}$. If $n$ is odd, we conclude that $\lambda_{a}=1=\lambda_{a}'$. Moreover, we can define the character $\chi$ of $T$ by $\chi(a)=1$ and $\chi(b)=\varepsilon^{\ell}$ for some $\ell$ satisfying $2\ell\equiv \beta'-\beta\,(\operatorname{mod}\, n)$. Then, in this case, $\psi$ and $\psi'$ are always isomorphic, since $\varepsilon^{\beta'-\beta}=\chi(b)^{2}$.
        
        Now, suppose that $n$ is even and $\lambda_{a}=\lambda_{a}'$. Then, we can define $\chi\in \hat{T}$ satisfying $\varepsilon^{\beta'-\beta}=\chi(b)^{2}$ if and only if $\beta'-\beta$ is even. Therefore, any $\tau$-homogeneous involution on $\cd$ is isomorphic to $(\tau,-1,1)$, $(\tau,-1,\varepsilon)$, $(\tau,1,1)$ or $(\tau,1,\varepsilon)$. These four involutions are pairwise non-isomorphic. However, we will show that they are not pairwise non-equivalent. Since $n$ is even, we can write $n=2^{\alpha}q$ where $q$ is odd. Then, $\mathbb{Z}_{n}\cong\mathbb{Z}_{2^{\alpha}}\times\mathbb{Z}_{q}$. Therefore, $\F^{\sigma}\mathbb{Z}_{n}\cong\F^{\sigma'}\mathbb{Z}_{2^{\alpha}}\otimes\F^{\sigma''}\mathbb{Z}_{q}$ where $\sigma'$ and $\sigma''$ are the restrictions of $\sigma$ to $\mathbb{Z}_{2^{\alpha}}$ and to $\mathbb{Z}_{q}$ respectively. Then, by \Cref{remarkMoreGenralCase}, it is enough to analyse the case where $n$ is a power of $2$.

        For $n = 2$, we have that $(\tau,-1,1)$, $(\tau,1,1)$ and $(\tau,1,\varepsilon)$ are pairwise equivalent and that they are not equivalent to $(\tau,-1,\varepsilon)$ (see \cite[Proposition 2.53]{EK2013}). Then, assume that $n=2^{\alpha}$ with $\alpha>1$. First, one can see that $(\tau,-1,1)$, $(\tau,1,1)$ and $(\tau,1,\varepsilon)$ are orthogonal while $(\tau,-1,\varepsilon)$ is symplectic. Therefore, the latter is not equivalent to any of the previous ones. Let $\psi = (\tau,-1,1)$ and $\psi'=(\tau,1,1)$. One can show that $\lambda_{\alpha a + \beta b}'=\varepsilon^{\alpha\beta}$. Define $\phi=
        \begin{pmatrix}
            1            & 0 \\
            \frac{
            n}{2}  & 1
        \end{pmatrix}$ and $\chi\colon T\to\F^{\times}$ being the map such that 
        \begin{align*}
            \chi(\alpha a+\beta b)=
            \begin{cases}
                \varepsilon^{\frac{n}{4}}, \text{ if $\alpha$ is odd;}\\
                1, \text{ if $\alpha$ is even.}
            \end{cases}
        \end{align*}
        Then, $\frac{\chi(g+h)}{\chi(g)\chi(h)}=\frac{\sigma(\phi(g),\phi(h))}{\sigma(g,h)}$ for all $g,h\in T$ and
        \begin{align*}
            \lambda_{a}=-1=\chi(a)\chi(\tau(a))^{-1}\lambda_{\phi(a)}';\\
            \lambda_{b}=1=\chi(b)\chi(\tau(b))^{-1}\lambda_{\phi(b)}'.
        \end{align*}
        Therefore, by \Cref{propositionEquivalenceAntiautomorphisms}, $\psi$ and $\psi'$ are equivalent.

        Now, we will show that $\psi=(\tau,1,\varepsilon)$ can not be equivalent to $\psi'=(\tau,1,1)$. In fact, if they were equivalent, by \Cref{propositionEquivalenceAntiautomorphisms} and \Cref{remarkAutomorphismEquivalenceParticularCase}, we should have a map $\chi\colon T\to \F^{\times}$ and a matrix $\phi\in \operatorname{SL}_{2}(\mathbb{Z}_{n})$ such that $\varepsilon=\lambda_{b}=\frac{\chi(b)}{\chi(-b)}\lambda_{\phi(b)}'$, $\frac{\chi(g+h)}{\chi(g)\chi(h)}=\frac{\sigma(\phi(g),\phi(h))}{\sigma(g,h)}$ for all $g,h\in T$ and $\phi\tau\phi^{-1}=\tau$. Then, $\phi$ should have the form
        $$\phi=
        \begin{pmatrix}
            x & y\\
            z & t
        \end{pmatrix}$$
        with $xt=1$ and $y,z\in \{0,\frac{n}{2}\}$. Since $\lambda_{\alpha a + \beta b}'=\varepsilon^{\alpha\beta}$, we would have that $\lambda_{\phi(b)}'$ should be an even power of $\varepsilon$. However, $\chi(b)\chi(-b)^{-1}$ is also an even power of $\varepsilon$. In fact, $\chi(b)$ is an $n$-th root of unity and
        \begin{align*}
            \frac{\chi(b)}{\chi(-b)}=\chi(2b)\frac{\sigma(\phi(2b),\phi(b))}{\sigma(2b,b)}=\chi(2b)=\chi(b)^{2}\frac{\sigma(\phi(b),\phi(b))}{\sigma(b,b)}= \pm\chi(b)^{2}.
        \end{align*}
        Therefore, the equation $\varepsilon=\frac{\chi(b)}{\chi(-b)}\lambda_{\phi(b)}'$ is not possible.

        Now, suppose that $n$ is even and let $\tau=\theta_{2}$. Then, by \Cref{corollaryAntiautomorphism} and \Cref{corollaryInvolution}, $(\tau,\lambda_{a},\lambda_{b})$ is a $\tau$-homogeneous involution if and only if $\lambda_{a}$ is an $n$-th root of unity and $\lambda_{b}=\lambda_{a}^{-1}$. Let $\psi=(\tau,\lambda_{a},\lambda_{a}^{-1})$ and $\psi'=(\tau,\lambda_{a}',\lambda_{a}'^{-1})$ be two $\tau$-homogeneous involutions. By \Cref{propIsomorphismAntiAuto}, $\psi$ and $\psi'$ are isomorphic if, and only, if there exists $\chi\in \hat{T}$ such that $\lambda_{a}'\lambda_{a}^{-1}=\chi(a)\chi(b)^{-1}$. Define the character $\chi$ of $T$ by $\chi(a)=\lambda_{a}'$ and $\chi(b)=\lambda_{a}$. Notice that $\lambda_{a}'\lambda_{a}^{-1}=\chi(a)\chi(b)^{-1}$. Then, $\psi$ and $\psi'$ are always isomorphic.

        Now, suppose that $4\mid n$ and let $\tau=\theta_{3}$. Then, by \Cref{corollaryAntiautomorphism} and \Cref{corollaryInvolution}, $(\tau,\lambda_{a},\lambda_{b})$ is a $\tau$-homogeneous involution if and only if $\lambda_{a}^{2}=-1=\lambda_{b}^{\frac{n}{2}}$. Let $\psi=(\tau,\lambda_{a},\lambda_{b})$ and $\psi'=(\tau,\lambda_{a}',\lambda_{b}')$ be two $\tau$-homogeneous involutions. Then, $\lambda_{a},\lambda_{a}'\in \{\varepsilon^{-\frac{n}{4}},\varepsilon^{\frac{n}{4}}\}$, $\lambda_{b}=\varepsilon^{2\beta+1}$ and $\lambda_{b}'=\varepsilon^{2\beta'+1}$ for some $\beta,\beta'\geq0$. By \Cref{propIsomorphismAntiAuto}, $\psi$ and $\psi'$ are isomorphic if, and only, if there exists $\chi\in\hat{T}$ such that $\chi(b)^{\frac{n}{2}}\lambda_{a}'=\lambda_{a}$ and $\chi(a)^{2}\lambda_{b}'=\chi(b)^{2}\lambda_{b}$. If $\lambda_{a}=\lambda_{a}'$, define $\chi\in\hat{T}$ by $\chi(a)=\varepsilon^{\beta-\beta'}$ and $\chi(b)=1$. If $\lambda_{a}=-\lambda_{a}'$, then define $\chi\in\hat{T}$ by $\chi(a)=\varepsilon^{1+\beta-\beta'}$ and $\chi(b)=\varepsilon$. In both cases, we conclude that $\chi(b)^{\frac{n}{2}}\lambda_{a}'=\lambda_{a}$ and $\chi(a)^{2}\lambda_{b}'=\chi(b)^{2}\lambda_{b}$. Therefore, all $\tau$-homogeneous involutions are isomorphic.
    \end{proof}
    
    \section{Homogeneous involutions on matrix algebras}
    In this section, we will be interested in $\tau$-homogeneous involutions on full matrix algebras over an algebraically closed field $\F$ of characteristic different from $2$ such that $\tau$ can be extended to an involution of $G$. The construction is similar to the ordinary case \cite{BK10,E10} (see also \cite[Section 2.4]{EK2013}) and of the degree-inverting case \cite{FSY22}. 

    Let $\mathcal{D}$ be a  $G$-graded-division algebra with support $T\subseteq G$ and let $0\neq V$ be a $G$-graded right $\mathcal{D}$-module finite-dimensional over $\cd$. The \textit{dual} of $V$ is the left $\mathcal{D}$-module $V^{\ast}=\mathrm{Hom}_{\mathcal{D}}(V,\mathcal{D})$ with the natural grading. For $f\in V^{\ast}$ and $v\in V$, we write $\langle f,v\rangle=f(v)$ to emphasize the duality between $V$ and $V^{\ast}$. Also, one has
    $$\deg\langle f,v\rangle=(\deg f)(\deg v)\quad \forall f \in V^{\ast},v \in V.$$

    Let $\mathcal{R}=\mathrm{End}_{\cd}(V)$. Then, $\mathcal{R}$ is a matrix algebra with entries in $\cd$ endowed with a $G$-grading. Notice that $V$ is a graded left $\mathcal{R}$-module with the natural action. Moreover, $V^{\ast}$ become a graded right $\mathcal{R}$-module with the action given by
    \begin{equation*}
        \langle fr,v\rangle=\langle f,rv\rangle,\quad f \in V^{\ast},\, r \in \mathcal{R},\, v \in V.
    \end{equation*}
    
    Now, let $0\neq W$ be a graded $\mathcal{R}$-submodule of $V^{\ast}$. Then, $W^{\perp}=\{v\in V\mid\,\forall f \in W,\, \langle f,v\rangle=0\}$ is a proper graded $\mathcal{R}$-submodule of $V$. Thus, $W^{\perp}=0$, which implies $W=V^{\ast}$. Therefore, $V^{\ast}$ is a graded-simple right $\mathcal{R}$-module.

    Assume that $\tau\colon G\to G$ is a bijection and $\psi$ is a $\tau$-homogeneous anti-automorphism of $\mathcal{R}$. Then, $V^{\ast}$ becomes a left $\mathcal{R}$-module by
    $$r\cdot f=f\psi(r),\quad f \in V^{\ast},\, r \in \mathcal{R}.$$
    In fact, $V^{\ast}$ with this action is a graded left $\mathcal{R}$-module if and only if $\tau$ is the identity map and $G$ is abelian. However, the following definition and lemma show how the structure of a left $\mathcal{R}$-module in $V^{\ast}$ interacts with the grading.
    
    \begin{definition}
        We say that a left $\mathcal{R}$-module $V$ is $\tau$-\emph{graded} if
        $$\mathcal{R}_{g}V_{s}\subseteq V_{s\tau(g)},\quad \forall g,s \in G.$$
    \end{definition}
    
    \begin{lemma}
        $V^{\ast}$ is a left $\tau$-graded $\mathcal{R}$-module with the action
        $$r \cdot f=f\psi(r),\quad f \in V^{\ast},\, r \in \mathcal{R}.$$
    \end{lemma}
    \begin{proof}
        Let $r \in \mathcal{R}_{g},\, f \in V^{\ast}_{s},\, v \in V_{h}$. Then
        \begin{align*}
            \deg(r\cdot f)h=\deg\langle r\cdot f,v\rangle=\deg \langle f\psi(r),v\rangle= \deg \langle f,\psi(r)v\rangle = s\tau(g)h.
        \end{align*}
        Therefore, $\mathcal{R}_{g}V^{\ast}_{s}\subseteq R_{s\tau(g)}$.
    \end{proof}

    Given a $G$-graded vector space $W=\bigoplus _{g\in G}W_{g}$, we define $W^{\tau}=\bigoplus _{g\in G}W^{\tau}_{g}$ where $W^{\tau}_{g}= W_{\tau(g)}$ for all $g \in G$.

    \begin{lemma}
        \label{antiautoRemark}
        Let $\tau\colon G\to G$ be an anti-automorphism of $G$ and let $W=\bigoplus _{g\in G}W_{g}$ be a $G$-graded vector space. Then, $W$ is a $\tau$-graded left $\mathcal{R}$-module if and only if $W^{\tau}$ is a graded left $\mathcal{R}$-module.
    \end{lemma}
    \begin{proof}
        If $W$ is a $\tau$-graded left $\mathcal{R}$-module, then:
        \begin{equation*}
            \mathcal{R}_{g}W^{\tau}_{s}=\mathcal{R}_{g}W_{\tau(s)}\subseteq W_{\tau(s)\tau(g)}= W_{\tau(gs)}=W^{\tau}_{gs}.
        \end{equation*}
        Conversely, if $W^{\tau}$ is a graded left $\mathcal{R}$-module,
        \begin{equation*}
            \mathcal{R}_{g}W_{s}=\mathcal{R}_{g}W^{\tau}_{\tau^{-1}(s)}\subseteq W^{\tau}_{g\tau^{-1}(s)}= W_{\tau(g\tau^{-1}(s))}= W_{s\tau(g)}.
        \end{equation*}
    \end{proof}

    \begin{lemma}
        \label{VastGradedSimple}
        Let $\tau\colon G\to G$ be an anti-automorphism. Then, $(V^{\ast})^{\tau}$ is a graded-simple left $\mathcal{R}$-module.
    \end{lemma}
    \begin{proof}
        It follows from \Cref{antiautoRemark} and from the fact that $V^{\ast}$ is a graded-simple right $\mathcal{R}$-module.
    \end{proof}

    By \Cref{VastGradedSimple} and \cite[Lemma 2.7]{EK2013}, we have the following lemma.

    \begin{lemma}
        \label{phi1}
        Let $\tau\colon G\to G$ be an anti-automorphism. There exists an element $g_{0}\in G$ and an $\mathcal{R}$-isomorphism $\varphi_{1}\colon V^{[g_{0}]}\to V^{\ast}$ satisfying
        $$\deg\varphi_{1}(v)=\tau(\deg v),\quad v\in V^{[g_{0}]}.$$
        Equivalently, $\varphi_{1}\colon V^{[g_{0}]}\to (V^{\ast})^{\tau}$ is a graded $\mathcal{R}$-isomorphism.\qed
    \end{lemma}

    From now on, suppose that $\tau$ is an anti-automorphism of $G$ and fix $g_{0}\in G$ and $\varphi_{1}$ as in \Cref{phi1}.

    \begin{lemma}
        There exists a homogeneous anti-automorphism $\psi_{0}\colon \mathcal{D}\to \mathcal{D}$ such that
        $$\varphi_{1}(vd)=\psi_{0}(d)\varphi_{1}(v).$$
        Furthermore, $\deg\psi_{0}(d)=\tau(g_{0})\tau(\deg d)\tau(g_{0})^{-1}$.
    \end{lemma}
    \begin{proof}

        
        We will first prove that the following sets coincide:
        \begin{align*}
            S_{1}&=\{\varphi\colon V^{[g]}\to (V^{\ast})^{\tau} \text{ graded $\mathcal{R}$-isomorphism for some $g \in G$}\},\\
            S_{2}&=\{\varphi_{1}\circ R_{d}\mid 0\neq d \in \mathcal{D}^{\times} \text{ homogeneous}\},\\
            S_{3}&=\{L_{d}\circ\varphi_{1}\mid 0\neq d \in \mathcal{D}^{\times} \text{ homogeneous}\},
        \end{align*}
        where $R_{d}\colon V\to V$ and $L_{d}\colon V^{\ast}\to V^{\ast}$ 
  denote right and left multiplication by $d$, respectively.
        
        Note that $S_{2},S_{3}\subseteq S_{1}$. Now, given $\varphi \in S_{1}$, we have that $\varphi_{1}^{-1}\circ\varphi \in \mathrm{End}_{\mathcal{R}}^{gr}(V) \cong \mathcal{D}$ is a graded map. Thus, there exists $d \in \mathcal{D}$ homogeneous such that $\varphi=\varphi_{1}\circ R_{d}\in S_{2}$. Analogously, $\varphi\circ \varphi_{1}^{-1} \in \mathrm{End}_{\mathcal{R}}^{gr}(V^{\ast}) \cong \mathcal{D}$ is a graded map. Then, we can find $d \in \mathcal{D}$ homogeneous such that $\varphi=L_{d}\circ \varphi_{1}\in S_{3}$.

        Since $S_{2}=S_{3}$, we can define $\psi_{0}\colon \mathcal{D}\to\mathcal{D}$ a linear isomorphism such that
        \begin{equation}
            \label{eqpsi0}
            \varphi_{1}\circ R_{d}=L_{\psi_{0}(d)}\circ\varphi_{1}
        \end{equation} 
        for any homogeneous element $d \in \mathcal{D}$. Moreover, by \eqref{eqpsi0} we have that $\psi_{0}$ is an anti-automorphism and $\varphi_{1}(vd)=\psi_{0}(d)\varphi_{1}(v)$, for all $v \in V$. Thus,
        $$\tau(\deg(v)\deg(d)g_{0})=\deg\psi_{0}(d)\tau(\deg(v)g_{0})\quad\forall v \in V,\,\forall d\in \mathcal{D}.$$
        Therefore, $\deg\psi_{0}(d)=\tau(g_{0})\tau(\deg d)\tau(g_{0})^{-1}$.
    \end{proof}
    
Defining $B \colon V \times V \to \cd$, as
$$B (v,w) = \langle \varphi_1(v), w \rangle,$$
we have a non-degenerate $\F$-bilinear form which satisfies the following properties:
\begin{itemize}
    \item[(\textrm{i})] $\deg B(v,w) = \tau (g_0) \tau (\deg v) \deg w$, for all homogeneous $v,w \in V$;
    \item[(\textrm{ii})] $B (vd,w) = \psi _0(d) B(v,w)$ and $B(v,wd) = B(v,w)d$, for every $v,w \in V$ and $d \in \cd$;
    \item[(\textrm{iii})] $B(rv,w) = B(v, \psi(r)w)$, for all $v, w \in V$ and $r \in \mcr$.
\end{itemize}

Now, let $\{ w_1, \ldots , w_n\}$ be a homogeneous $\cd$-basis of $V$. Denote by $\Phi$ the matrix of $B$, i.e., the matrix with entries $x_{ij} = B(w_i,w_j)$. Given $r \in \mcr$, let $R = (r_{ij})$ be its matrix form, and $\psi (R) = (r'_{ij})$ be the matrix form of $\psi(r)$. Then,
$$B(rw_k,w_\ell) = B \left( \displaystyle \sum _{i=1}^n w_ir_{ik}, w_\ell\right)= \sum _{i=1}^n \psi_0(r_{ik})x_{i\ell},$$
and
$$B(w_k,\psi(r)w_\ell) = B \left( w_k,  \displaystyle \sum _{i=1}^n w_ir'_{i\ell}\right)=  \sum _{i=1}^n x_{ki}r'_{i\ell}.$$

Hence, we obtain $\psi_0(R)^t \Phi = \Phi R$. Then, identifying $\mcr=M_{n}(\cd)$ and defining $\psi_0(X)  = (\psi_0(x_{ij}))$, we have
\begin{equation} \label{psimatrix}
    \psi: X \in \mcr \mapsto \Phi ^{-1}\psi _0(X^t)\Phi \in \mcr,
\end{equation}
where $^t$ is the usual transpose involution on $M_{n}(\cd)$. Assuming that $\cd$ is a matrix algebra and $\psi_{0}$ is the usual matrix transposition on $\cd$, we can get rid of $\psi_{0}$ in \eqref{psimatrix}. In this case, $^{t}$ should be understood as the usual transposition on $M_{n}(\F)\otimes\cd$.

This shows that, given a pair $(B, \psi_0)$ satisfying $\mathrm{(i)-(iii)}$, we can recover uniquely the anti-automorphism $\psi$.

\begin{proposition}
  Let $G$ be a group, $\cd$ a graded-division algebra, $V$ a graded right $\cd$-module finite-dimensional over $\cd$ and $\mcr = \operatorname{End}_\cd (V)$. 
  Assume that $\psi$ is a $\tau$-homogeneous anti-automorphism on $\mcr$, where $\tau : G \to G$ is an anti-automorphism. Then, there exist $g_0 \in G$, an anti-automorphism $\psi_0$ on $\cd$ satisfying $\deg \psi_0(d) = \tau(g_0)\tau(\deg d)\tau(g_0)^{-1}$ for all homogeneous $d \in \cd$,
  and a non-degenerate form $B: V \times V \to \cd$ satisfying $\mathrm{(i)-(iii)}$. If $(\psi_0', B')$ is another such pair, then there exists a nonzero homogeneous $d \in \cd$ such that $B' = dB$ and $\psi_0'(x) = d \psi _0(x)d^{-1}$, for every $x \in \cd$.

  Conversely, given a pair $(\psi_0, B)$ satisfying $\mathrm{(i)-(iii)}$, there exists a $\tau$-homogeneous anti-automorphism on $\mcr$.\qed
\end{proposition}

From now on, we assume that $\cd$ is finite-dimensional and $\psi$ is a $\tau$-homogeneous involution. Then, $\psi_{0}$ has to be an involution on $\cd$ and, from \Cref{lemmaTau}, $\tau$ has to be an involution on $G$. Our next step is to construct a special $\cd$-basis of $V$ to simplify the matrix $\Phi$ in \eqref{psimatrix}. The following result can be proved by following \textit{verbatim} the proof of Lemma 4.6 in \cite{FSY22}.

\begin{lemma} \label{epsilon+-1}
    Let $\psi: \mcr \to \mcr$ be a $\tau$-homogeneous involution. Then
    $$B(w,v) = \varepsilon _B \psi_0\left(B(v,w)\right), \quad \forall \, v,  w \in V,$$
    where $\varepsilon_B \in \{-1,1\} $.
\end{lemma}

As a consequence, we conclude that $B$ is a balanced bilinear form, that is, $B(u, v) = 0$ if and only if $B(v, u) = 0$. If $U \subseteq V$ is a $\cd$-subspace, we define
$$U^\perp = \{x \in V \, | \, B(U,x)=0 \} =  \{x \in V \, | \, B(x,U)=0 \}.$$

A standard verification yields that $V = U \oplus U ^\perp$ if and only if $B|_{U}$ is non-degenerate.
Then, we can construct a homogeneous $\cd$-basis $\{v_1, \ldots , v_m, v'_{m+1}, \ldots , v_s', v''_{m+1},\ldots, v''_s\}$ of $V$ such that
\begin{itemize}
    \item $B(v_i, v_i) \neq 0$, for $i = 1, \ldots , m$;
    \item $B(v'_j, v''_j)=1$, for $j>m$;
    \item $B(v,w)=0$ in the remaining cases.
\end{itemize}

For $1\leq i\leq m$, let $g_i = \deg v_i$ and $t_{i}=\deg B(v_{i},v_{i})$. For $m< j\leq s$, let $g_j' = \deg v'_j$ and $ g_j'' = \deg v_j''$. 
Then,
$$t_{i} = \deg B(v_i, v_i) =  \tau (g_0) \tau(g_{i})g_{i}.$$
Therefore,
$$t_{i} = \deg B(v_i, v_i) = \deg \psi_{0}(B(v_i, v_i)) = t_{i}g_{0}\tau(g_{0})^{-1}.$$
Moreover,
$$1 = \deg B(v'_j, v''_j) =  \tau (g_0) \tau(g_j')g_j''.$$
Thus, $g_{j}'' = \tau(g_{j}')^{-1}g_{0}^{-1}$ and
$$1 =  \deg B(v'_j, v_j'')= \deg \psi_0(B(v_j'', v'_j)) =g_0 \tau (g_0)^{-1}. $$
Hence, we obtain that always $g_{0}=\tau(g_{0})$.

    We can note that $\psi$ is orthogonal if $\varepsilon_{B}=1$ and that $\psi$ is symplectic otherwise. Moreover, $\varepsilon_{B}=-1$ implies $m=0$ in the previous notation. We summarize our discussion above.

\begin{theorem}
    Let $G$ be a group, $\tau$ be an involution on $G$ and $\F$ be an algebraically closed field of characteristic different from $2$. Let $\cd$ be a finite-dimensional graded-division algebra with support $T\subseteq G$ and let $\mathcal{R}=M_{n}(\cd)$ endowed with a $G$-grading. Then, $\mathcal{R}$ admits a $\tau$-homogeneous involution $\psi$ if and only if there exists an element $g_{0}\in G$, an involution $\psi_{0}$ on $\cd$, a sequence
    $$\gamma=(g_{1},\ldots,g_{m},g_{m+1}',\ldots,g_{s}',g_{m+1}'',\ldots,g_{s}'')$$
    of elements from $G$ and a sequence
    $$\gamma=(t_{1},\ldots,t_{m})$$
    of elements from $T$ such that $\deg(\psi_{0}(d))=\tau(g_{0})\tau(\deg d)\tau(g_{0})^{-1}$ for all $d\in \cd$ homogeneous, $t_{i}=\tau(g_{0})\tau(g_{i})g_{i}$ for all $1\leq i\leq m$, $g_{j}'' = \tau(g_{j}')^{-1}g_{0}^{-1}$ for all $j > m$ and $\mcr$ is isomorphic to $\mathcal{M}(\cd,\gamma)$. In this case, $g_0 = \tau (g_0)$. Moreover, if $\varphi\colon \mcr\to \mathcal{M}(\cd,\gamma)$ is an isomorphism and $\{X_{u}\mid u\in T\}$ is a basis of $\cd$, then $\psi'=\varphi\psi\varphi^{-1}$ is a $\tau$-homogeneous involution on $\mathcal{M}(\cd,\gamma)$ given by $\psi'(e_{ij}\otimes X) = \Phi ^{-1} e_{ji}\otimes\psi_{0}(X) \Phi$, where
    $$\Phi = \displaystyle \sum _{i=1}^m 1 \otimes X_{t_i} \oplus \left( 
    \begin{matrix}
        0 & I_s\\
        I_{s} & 0
    \end{matrix}
    \right)\otimes X_{1}, \text{ if $\psi$ is orthogonal}$$
    or
    \[
        \pushQED{\qed} 
        \Phi = \left( 
        \begin{matrix}
            0 & I_s \\
            -I_s & 0
        \end{matrix}
        \right) \otimes X_1, \text{ if $\psi$ is symplectic}.\qedhere
        \popQED
    \]     
\end{theorem}

\section*{Acknowledgements}
    We thank F. Y. Yasumura for suggesting the topic, for the insightful discussions and for carefully reading and providing feedback that improved this paper.

\bibliographystyle{abbrv}

\bibliography{Refs.bib}

\end{document}